\documentclass{amsart}
\usepackage{amsmath, enumerate}
\usepackage{algorithm}
\usepackage{algpseudocode}
\usepackage{amsfonts}
\usepackage{amssymb}
\usepackage{amsthm}
\usepackage{color}
\usepackage{caption}
\usepackage{graphicx}
\usepackage{url}
\usepackage{verbatim}

\newcommand{\Z}{\mathbb{Z}}

\newcommand{\Q}{\mathbb{Q}}

\newcommand{\R}{\mathbb{R}}

\newcommand{\bp}{\begin{problem}}

\newcommand{\ep}{\end{problem}}

\newcommand{\ba}{\begin{answer}}

\newcommand{\ea}{\end{answer}}

\newcommand{\ben}{\renewcommand{\theenumi}{\alph{enumi}}

\renewcommand{\labelenumi}{(\theenumi)}\begin{enumerate}}

\newcommand{\een}{\end{enumerate}}

\newtheorem{defin}{Definition}[section]
\newtheorem{thm}[defin]{Theorem}
\newtheorem{cor}[defin]{Corollary}
\newtheorem{rem}[defin]{Remark}
\newtheorem{lem}[defin]{Lemma}
\newtheorem{prop}[defin]{Proposition}
\newtheorem{ex}[defin]{Example}

\DeclareCaptionType{result}[Table][List of Results]
\newenvironment{results}{\captionsetup{type=result}\center}{}

\title[On the Conjugacy Search Problem in Certain Metabelian Groups]{On the Conjugacy Problem in Certain Metabelian Groups}
\begin{document}

\author[J. Gryak]{Jonathan Gryak}
\address{Jonathan Gryak, Department of Computational Medicine and Bioinformatics, University of Michigan, Ann Arbor} \email{gryakj@med.umich.edu}

\author[D. Kahrobaei]{Delaram Kahrobaei}
\address{Delaram Kahrobaei, CUNY Graduate Center, PhD Program in Computer Science and NYCCT, Mathematics Department, City University of New York}
\email{dkahrobaei@gc.cuny.edu}

\author[C. Martinez-Perez]{Conchita Martinez-Perez}
\address{Conchita Martinez-Perez, Departamento de Matem\'aticas, University of Zaragoza, Spain}
\email{conmar@unizar.es}

\begin{abstract}
We analyze the computational complexity of an algorithm to solve the conjugacy search problem in a certain family of metabelian groups. We prove that in general the time complexity of the conjugacy search problem for these groups is at most exponential. For a subfamily of groups we prove that the conjugacy search problem is polynomial. We also show that for a different subfamily the conjugacy search problem reduces to the discrete logarithm problem.  
\end{abstract}

\maketitle
\tableofcontents

\section{Introduction}

The conjugacy decision problem is, along with the word and isomorphism
problems, one of the original group-theoretic decision problems introduced by Max Dehn
in 1911. In this paper we will consider a variant of the first called the conjugacy search problem, in which we are given two conjugate elements $g$ and $g_1$ of a group $G$ and are asked to find an element $h$ in $G$ such that $g^h=hgh^{-1}$.  There are groups for which the conjugacy decision problem is unsolvable, whereas the search variant is always solvable. In the case of finitely generated metabelian groups, the conjugacy decision problem is solvable \cite{noskov1982}. 

\

The original motivation for the conjugacy search problem comes from group-based cryptography. In this paper our focus is on the algorithmic aspects of the conjugacy search problem, rather than its potential use in cryptographic applications. In particular, we develop and estimate the computational complexity of an algorithm that solves the conjugacy search problem in a certain family $\mathcal{F}$ of finitely presented metabelian groups (see Section \ref{groups}). 
Note that in the literature there are many algorithmic results available for the conjugacy decision, conjugacy search, and other group-theoretic problems, but primarily for polycyclic groups (\cite{eickostheimer2003}, \cite{handbook}). However, groups in our family need not to be polycyclic. To rectify this, we essentially translate the existing algorithms for metabelian polycyclic groups to a non-polycyclic setting. There are a number of technical difficulties to achieve this, and a large part of the paper is devoted to solving these difficulties. 

\
  
A particular subfamily of $\mathcal{F}$ consists of the following generalization of Baumslag-Solitar groups: 
$$G=\langle q_1,\ldots,q_n,b\mid [q_l,q_t]=1,q_lbq_l^{-1}=b^{m_l},1\leq l,t\leq n\rangle,$$
where $m_1,\ldots,m_n$ are integer numbers.
Observe that such generalized Baumslag-Solitar groups split as semidirect products $B\rtimes Q$ with $Q$ free abelian of finite rank $n$ and $B=\Z[{1\over p_1},\ldots,{1\over p_n}]$.
 In general, groups in $\mathcal{F}$ split as a semidirect product $G=B\rtimes Q$ with $Q$ free abelian of finite rank and where $B$ can be seen as an additive subgroup of $\Q^s$ for some $s$.    Throughout the paper, we will consider $B$ as a $Q$-module with left action.  Upon fixing a basis, the action of $Q$ on $B$ can be described using integral matrices
that commute pairwise. 
Such a group $G$ is polycyclic if and only if the matrices encoding the action of $Q$ have integral inverses \cite{AuslanderHall}.

\

Groups in $\mathcal{F}$ enjoy strong finiteness properties, for example they have finite Pr\"ufer rank (meaning that the number of generators needed to generate any finitely generated subgroup is bounded),
have cohomological type $\text{FP}_\infty$ \cite[Proposition 1]{baumslag1976constructable} (see also the proof of Theorem 8 in the same paper) and are constructible  (i.e., can be constructed in finitely many steps from the trivial group using finite index extensions and ascending HNN-extensions). In fact, our groups are iterated, strictly ascending HNN-extensions of a free abelian group of finite rank. Moreover, any constructible torsion-free split metabelian group of finite Pr\"ufer rank lies in $\mathcal{F}$ and any metabelian group of finite Pr\"ufer rank can be embedded in a group in $\mathcal{F}$ \cite{baumslag1976constructable}.

\

The general strategy of our algorithm consists of using linear representations of our groups so we can utilize known methods from linear algebra, such as the polynomial time solution of the multiple orbit problem given in  \cite{aszl1996multiplicative}. In order to uses these methods we need an efficient means of swapping between the representation of an element of $B$ inside $\Q^s$ and its word representation. After giving the precise definition of the family $\mathcal{F}$ and fixing some notation, we develop the necessary techniques to swap between these representations in Section \ref{groups}. These include a method to decide whether some arbitrary vector in $\Q^s$ belongs to $B$, as well as the ability to determine when a system of linear equations has some solution in $B$. Along the way we show that the computational complexity of these procedures is sufficiently low to switch between word and linear representations as needed.

\

In Section \ref{complexity}, we describe and analyze the computational complexity of our algorithm for the conjugacy search problem in $\mathcal{F}$. In particular, we prove the following theorem:
\begin{thm}\label{exponential} For any $G\in\mathcal{F}$, the time complexity of the conjugacy search problem for conjugate elements $g,g_1\in G$ is at most exponential in the length of $g$ and $g_1$. 
\end{thm}
As a corollary, we also deduce some consequences about conjugator lengths (Corollary \ref{length}).

\

Of course, the complexity bound of Theorem \ref{exponential} is too large for some particular choices of the group $G$. This is the case if $G$ happens to be nilpotent, as the complexity is known to be polynomial (see \cite{sims1994}, we thank the anonymous referee for pointing out this fact). We make this explicit by giving an example of a subfamily of $\mathcal{F}$ consisting of nilpotent groups where our algorithm is also polynomial (see Theorem \ref{polynomial}).

\
 
However, there are some particular cases of groups in $\mathcal{F}$ for which the conjugacy search problem reduces to a type of discrete logarithm problem. This is discussed in Subsection \ref{discretelog}. In particular, this is the case for generalized metabelian Baumslag-Solitar groups given by a relatively simple presentation of the form 
$$
G=\left\langle q_1,q_2,b|b^{q_1}=b^{m_1},b^{q_2}=b^{m_2},[q_1,q_2]=1\right\rangle.
$$

\section{Split Metabelian Groups of Finite Pr\"ufer Rank}\label{groups}


\subsection{The Family $\mathcal{F}$ and Notation} Consider the following group presentation
\begin{equation}\label{F}G=\langle q_1,\ldots,q_n,b_1,\ldots,b_s\mid [q_l,q_t]=1,[b_i,b_j]=1,\mathcal{R}\rangle, \hbox{ with}\end{equation}
$$\mathcal{R}=\{q_lb_iq_l^{-1}=b_1^{m_{l(1,j)}}b_2^{m_{l(2,j)}}\ldots b_s^{m_{l(s,j)}};m_{l(i,j)}\in\Z\},$$
where we require the following extra condition: for $l=1,\ldots,n$ let $M_l$ be the integer matrix encoding the action of $q_l$, that is, the $s\times s$ matrix with $j$-th column $m_{l(1,j)},\ldots,m_{l(s,j)}$. Then the matrices $M_l$ have to commute pairwise.

\

We define the class $\mathcal{F}$ as the class of groups admitting a presentation as (\ref{F}). For a fixed group $G$ with presentation (\ref{F}) we denote by $Q$ the subgroup of $G$ generated by $q_1,\ldots,q_n$ and by $B$ the normal subgroup of $G$ generated by $b_1,\ldots,b_s$. Note that elements in $B$ are precisely those elements expressed by words where the exponent sum of instances of every $q_l$ is zero. From this fact one easily sees that $G=B\rtimes Q$.
We use multiplicative notation for the whole group $G$ but often we use additive notation for $B$. If $c\in B$, $x\in Q$, the conjugation action of the element $x$ on $c$ is denoted
$x\cdot c$ additively or $c^x=xcx^{-1}$ multiplicatively.

Any element in  $G$ 
can be represented by a word of the following type
\begin{equation}\label{form}q_1^{-\alpha_1}\ldots q_n^{-\alpha_n}b_1^{\beta_1}\ldots b_s^{\beta_s}q_1^{\gamma_1}\ldots q_n^{\gamma_n},\end{equation}
with $\alpha_1,\ldots,\alpha_n\geq 0$ and such that whenever $\alpha_i\neq 0$, the element $q_i^{-1}b_1^{\beta_1}\ldots b_s^{\beta_s}q_i$ does not belong to the subgroup generated by $b_1,\ldots,b_s$. 
There is an efficient algorithm (collection) to transform any word in the generators into a word of the previous form representing the same group element: use the relators to move all of the instances of $q_i$ with negative exponents to the left and all the instances of $q_i$ with positive exponents to the right 
(see example \ref{genBS}).


\begin{ex}\label{genBS} Generalized Metabelian Baumslag-Solitar Groups. {\rm Let $m_1,\ldots,m_n$ be positive integers. We call the group given by the following presentation a {\sl generalized metabelian Baumslag-Solitar group} 
$$G=\langle q_1,\ldots,q_n,b\mid b^{q_l}=b^{m_l},\,  [q_l, q_j]=1,\,  1\leq i, j \leq n,\rangle.$$
It is a constructible metabelian group of finite Pr\"ufer rank and $G\cong B\rtimes Q$ with $Q=\langle q_1,\ldots,q_n\rangle\cong\Z^n$ and $B=\Z[m_1^{\pm1},\ldots,m_k^{\pm1}]$ (as additive groups). Let us examine how collection works for these groups. Consider for example the group
$$
G=\langle q_1,q_2,b\mid b^{q_1}=b^{2},b^{q_2}=b^{3},[q_1,q_2]=1\rangle,
$$
with $G\cong\Z\left[\frac{1}{2},\frac{1}{3}\right]\rtimes\Z^2$, and an uncollected word 
$q_2bq_1^{-1}q_2b^{-1}q_1$ in $G$
representing the group element $w$. An iterated use of the relations $q_2b^{\pm 1}=b^{\pm 3}q_2$ and $b^{\pm 1}q_1^{-1}=q_1^{-1}b^{\pm 2}$ together with the commutation relations between $q_1$ and $q_2$  yields
$$
w=q_2bq_1^{-1}q_2b^{-1}q_1=b^3q_2q_1^{-1}b^{-3}q_2q_1=q_1^{-1}b^6b^{-9}q^2_2q_1=q_1^{-1}b^{-3}q_1q^2_2.
$$
}
\end{ex}
\

\begin{ex}\label{galois}{\rm Let $L:\Q$ be a Galois extension of degree $n$ and fix an integral basis $\{u_1,\ldots,u_s\}$ of $L$ over $\Q$.
Then $\{u_1,\ldots,u_s\}$ freely generates the maximal order $\mathcal{O}_L$ as a $\Z$-module. 
Now, we choose integral elements, $q_1,\ldots,q_n$, generating a  free abelian multiplicative subgroup of $L-\{0\}$. Each $q_l$ acts on $L$ by left multiplication and using the  basis $\{u_1,\ldots,u_s\}$,  we may represent this action by means of an integral matrix $M_l$. Let $B$ be the smallest sub $\Z$-module of $L$ closed under multiplication with the elements $q_l$ and  $q_l^{-1}$  and such that $\mathcal{O}_L\subseteq B$, i.e.,
$$B=\mathcal{O}_L[q_1^{\pm 1},\ldots,q_n^{\pm1}].$$
We may then define $G=B\rtimes Q$, where the action of $Q$ on $B$ is given by multiplication by the $q_l$'s. The generalized Baumslag-Solitar groups of the previous example are a particular case of this situation when $L=\Q$. If the elements $q_l$ lie in $\mathcal{O}_L^\times$, which is the group of units of $\mathcal{O}_L$, then the group $G$ is polycyclic.}
 \end{ex}
 
\subsection{Linear Representations}\label{semidirectB} In this subsection, we consider again a group with a presentation as in (\ref{F}). We will show that the subgroup $B$ can be seen as an additive subgroup of $\Q^s$ and see how to swap between representation of its elements as words in the generators and as vectors in $\Q^s$. Recall that $B$ consists precisely of those elements of $G$ that are given by words $w$ so that for each $i=1,\ldots,n$ the exponent sum of $q_i$'s is zero. Using the collection process above, such an element $b$ can be represented by a word
$$q_1^{-\alpha_1}\ldots q_n^{-\alpha_n}b_1^{\beta_1}\ldots n_s^{\beta_s}q_1^{\alpha_1}\ldots q_n^{\alpha_n}.$$
Additively,
 $$b=(q_1^{-\alpha_1}\ldots q_n^{-\alpha_n})\cdot ({\beta_1}b_1+\ldots+ {\beta_s}b_s).$$
Then, define the map $B\hookrightarrow\Q^s$ by mapping $b$ to
\begin{equation}\label{matrixb} v=M_1^{-\alpha_1}\cdots M_n^{-\alpha_n}\begin{pmatrix} \beta_1\\ \vdots\\ \beta_s\\ \end{pmatrix}.\end{equation}
One easily checks that this is a well-defined injection
that can be seen as an explicit recipe to transform words representing elements in $B$ into the corresponding vector of $\Q^s$. We can do something similar for arbitrary elements $g\in G$:
assume that $g$ is given by a word as in (\ref{form})
$$q_1^{-\alpha_1}\ldots q_n^{-\alpha_n}b_1^{\beta_1}\ldots b_s^{\beta_s}q_1^{\gamma_1}\ldots q_n^{\gamma_n},$$
then the following word also yields $g$:
$$q_1^{-\alpha_1}\ldots q_n^{-\alpha_n}b_1^{\beta_1}\ldots b_s^{\beta_s}q_1^{\alpha_1}\ldots q_n^{\alpha_n}q_1^{\gamma_1-\alpha_1}\ldots q_n^{\gamma_n-\alpha_n}.$$
In the semidirect representation we have $g=vx$ with $x=q_1^{\gamma_1-\alpha_1}\ldots q_n^{\gamma_n-\alpha_n}$ and $v$ is as before. 
The complexity of the computation in (\ref{matrixb}) using Gaussian elimination for inverses, standard matrix multiplication, and efficient exponentiation is:
$$
O((n-1)[s^3+s^3\log\max_l(\alpha_l)+s^3\log\max_l(\gamma_l-\alpha_l)]+s^2+s^3).
$$
Now, consider the converse, in which we have $vx$ with $v$ a vector in $\Q^s$. In order to express $vx$ as a word as in (\ref{form})
we first 
describe a particular subset of $\Q^s$ that contains $B$. 
In the following discussion, we identify $B$ with its image in $\Q^s$ and the group generated by $b_1\ldots,b_s$ with  $\Z^s$.

\

For $1\leq l \leq n$, let $d_l$ be the smallest positive integer such that 
$d_lM_l^{-1}$is an integral matrix, i.e., $d_l$ is the lowest common denominator of the matrix entries $m_{l(s,i)}$. Let $d=\prod_ld_l$ (if $G$ is polycyclic, $d=1$). For any $v\in B$,
$$
d^{\alpha_1+\ldots+\alpha_n}v\in\Z^s
$$
thus $v\in\Z[\frac{1}{d}]^s$, in other words, we have 
$$B\subseteq\Z[\tfrac{1}{d}]^s\subset \Q^s.$$


\begin{rem}{\rm This implies that for any $v\in B$, if $t$ is be the smallest positive integer such that $d^tv$ lies in $\Z^s$, then $t$ is bounded by twice the length of $v$ as a word in  (\ref{form}).}
\end{rem}

$B$ can also be constructed from $\Z^s$ and $M=\prod_lM_l$. Observe that
$$
\Z^s\subseteq M^{-1}\Z^s\subseteq\ldots\subseteq M^{-j}\Z^s\subseteq M^{-j-1}\Z^s\subseteq\ldots\subseteq B
$$
and in fact $B=\cup_{j=0}^\infty M^{-j}\Z^s.$ To check this, note that any vector in $B$ has the form $M_1^{-\beta_1}\ldots M_n^{-\beta_n}u$ for some $u\in\Z^s$ and certain $\beta_1,\ldots,\beta_n\geq 0$. Let $\beta=\max\{\beta_1,\ldots,\beta_n\}$, then
$$M_1^{-\beta_1}\ldots M_n^{-\beta_n}u=M^{-\beta}M_1^{\beta-\beta_1}\ldots M_n^{\beta-\beta_n}u=M^{-\beta}w$$
where $w=M_1^{\beta-\beta_1}\ldots M_n^{\beta-\beta_n}v$ lies in $\Z^s$. Consequently, if $q=q_1\ldots q_n$, then the group $B\rtimes \langle q\rangle$ is a strictly ascending HNN extension of $\Z^s$.

\begin{lem}\label{intersection} There is some $\alpha$ depending on $G$ only such that for any $i$,
$$B\cap {1\over d^i}\Z^s\subseteq M^{-i\alpha}\Z^s.$$
Moreover $\alpha\leq s\log d$.
\end{lem}
\begin{proof} Consider first the case when $i=1$. We have $\Z^s\subseteq{1\over d}\Z^s$ and
$$\Z^s\subseteq M^{-1}\Z^s\cap{1\over d}\Z^s\subseteq\ldots\subseteq M^{-j}\Z^s\cap{1\over d}\Z^s\subseteq M^{-j-1}\Z^s\cap{1\over d}\Z^s\subseteq\ldots\subseteq{1\over d}\Z^s.$$
As the quotient of ${1\over d}\Z^s$ over $\Z^s$ is the finite group $\Z_d\times\ldots\times \Z_d$ of order $d^s$, this sequence stabilizes at some degree, say  $\alpha$. Then 
$$B\cap {1\over d}\Z^s=M^{-\alpha}\Z^s\cap {1\over d}\Z^s\subseteq M^{-\alpha}\Z^s$$
as desired. Moreover, we claim that stabilizes precisely at the first $\alpha$ such that
$$M^{-\alpha}\Z^s\cap{1\over d}\Z^s= M^{-\alpha-1}\Z^s\cap{1\over d}\Z^s.$$
To demonstrate, let $b\in M^{-\alpha-2}\Z^s\cap{1\over d}\Z^s$. Then 
$$Mb\in M^{-\alpha-1}\Z^s\cap{1\over d}\Z^s=M^{-\alpha}\Z^s\cap{1\over d}\Z^s$$
 thus $b\in M^{-\alpha-1}\Z^s\cap{1\over d}\Z^s=M^{-\alpha}\Z^s\cap{1\over d}\Z^s$. Repeating the argument implies that for all $\beta>\alpha$,
$$M^{-\alpha}\Z^s\cap{1\over d}\Z^s=M^{-\beta}\Z^s\cap{1\over d}\Z^s.$$
As a consequence,  $\alpha$ is bounded by the length of the longest chain of proper subgroups in $\Z_d\times\ldots\times \Z_d$, i.e., $\alpha\leq\log(d^s)=s\log d$. 

\

For the case of an arbitrary $i$ we argue by induction. Let $b\in B\cap  {1\over d^i}\Z^s$, then $db\in B\cap  {1\over d^{i-1}}\Z^s$ and by induction we may assume that $db\in M^{-(i-1)\alpha}\Z^s$, thus  $M^{(i-1)\alpha}db=v\in\Z^s$. Then
$${1\over d}v\in B\cap  {1\over d}\Z^s\subseteq M^{-\alpha}\Z^s.$$
Therefore 
$$M^{\alpha}M^{(i-1)\alpha} b={1\over d}M^{i\alpha} v\in\Z^s$$
and $b\in M^{-i\alpha}\Z^s$.
\end{proof}

It is easy to construct examples with $\alpha\neq 1$ for $\alpha$ as in Lemma \ref{intersection}:

\begin{ex}\label{exalpha}{\rm

Consider the group $G\in\mathcal{F}$ given by the following presentation:
\begin{align*}
G&=\langle q_1,q_2,q_3,b_1,b_2,b_3\mid [q_l,q_t]=1,[b_i,b_j]=1,\mathcal{R}\rangle\, \hbox{with}\\
&\mathcal{R}=\{b_1^{q_1}=b_1^2,b_2^{q_2}=b_2^4,b_3^{q_3}=b_3^{16},b_j^{q_l}=b_j\text{ for }j\neq l\}\\
\end{align*}
From the presentation above $s=3$. The linear representations of the $q_l$'s (and their product $M$) are then:
$$
M_1=\left[\begin{array}{ccc}
2 & 0 & 0\\
0& 1 & 0\\
0& 0 &1
\end{array}\right]
M_2=\left[\begin{array}{ccc}
1 & 0 & 0\\
0& 4 & 0\\
0& 0 &1
\end{array}\right]
M_3=\left[\begin{array}{ccc}
1 & 0 & 0\\
0& 1 & 0\\
0& 0 &16
\end{array}\right];
M=\left[\begin{array}{ccc}
2 & 0 & 0\\
0& 4 & 0\\
0& 0 &16
\end{array}\right].
$$
From visual inspection of $M$ it is clear that $d=16$. Moreover, it is easy to check that ${1\over 16}\Z^s\subseteq B$ and that in fact
$${1\over 16}\Z^s={1\over 16}\Z^s\cap B\subseteq M^{-4}\Z^s.$$
Moreover $4$ is smallest possible in these conditions thus $\alpha=4$.
}\end{ex}

In determining whether a vector $v\in\Q^s$ lies in $B$, it is clear from the previous discussion that a necessary condition is that $v$ belongs to $\Z[{1\over d}]^s$, and therefore there exists a $t>0$ such that $v\in {1\over d^t}\Z^s$. In the particular case when $v$ is integral, then $v\in B$ and the coordinates of $v$ are the exponents of the $b_j$'s in an expression for $v$ as in (\ref{form}).

\

If $v$ is strictly rational, we can perform the following procedure to check whether $v\in{1\over d^t}\Z^s$ for some $t$ and to find the smallest possible such $t$. First, compute the least common multiple of the denominators of the entries of $v$. By reducing if necessary, we may assume that
$$v={1\over m}(v_1,\ldots,v_s)$$
with the $v_j$ integers so that no prime divides all of $m,v_1,\ldots, v_s$ simultaneously. Note that for an arbitrary $t$, $d^tv$ is integral if and only if $m$ divides $d^tv_j$ for each $j=1\ldots,n$, and by the choice of $m$ and $v_1,\ldots,v_s$ this is equivalent to $m$ dividing $d^t$. Therefore to decide whether $v$ lies in ${1\over d^t}\Z^s$ for some $t$ we 
only have to check whether $d^t=0$ modulo $m$ for some $t\geq 0$. Moreover, it is easy to check that if there is such a $t$ then there is some so that $t\leq m$, in other words, 
if explicit factorizations of $m$ and $d$ are not available, we need only compute $d^t$ for $1\leq i\leq m$ and in the case when none of these values is a multiple of $m$ we conclude that $v$ does not belong to $\Z[{1\over d}]^s$.

\begin{lem}\label{vinB} Let  $v\in\Z[{1\over d}]^s$ and  $t$ the smallest possible integer such that $d^tv$ is integral. Then $v\in B$ if and only if  $$M^{ts\lfloor\log d\rfloor}v\in\Z^s$$
where $M=M_1M_2\ldots M_n$.
The complexity of this computation is polynomial, specifically $O((n-1) s^3\log ts\lfloor\log d\rfloor)$. (Alternatively, the same result holds true but with $\alpha$ instead of $s\lfloor\log d\rfloor$).\end{lem}
\begin{proof}  Lemma \ref{intersection} implies that $v\in B$ if and only if $M^{t\alpha}v$ is integral. Thus if $v\in B$,
$$M^{ts\lfloor\log d\rfloor}v=M^{(ts\lfloor\log d\rfloor-t\alpha)} M^{t\alpha}v$$
is integral because $ts\lfloor\log d\rfloor-t\alpha\geq 0$. The converse is obvious.

\

Regarding the time complexity, we have to compute the $(ts\lfloor\log d\rfloor v)$-th power of the matrix $M$. The complexity estimation is obtained using standard matrix multiplication and efficient exponentiation. 
\end{proof}

\begin{rem}{\rm Note that the exponent $ts\lfloor\log d\rfloor$ is just an upper bound and often a much smaller value suffices to obtain an expression of a given $v\in B$ as a word on the generators of $G$.  Consider for example the group of Example \ref{exalpha} and  the vector $v\in\Q^3$:
$$
v'=\left[ \frac{1}{32},\frac{3}{64},\frac{5}{16}\right].
$$
Here, $t=2$, $s=3$, $d=16$ and $d=$ thus  $ts\lfloor\log d\rfloor=24$ but note that already $M^5v$ is integral.
}\end{rem}

\subsection{Solving Linear Systems} To finish this section and for future reference, we are going to consider the following problem. Assume that we have a square $s\times s$ integral matrix $N$ that commutes with all the matrices $M_l$ and a rational  column vector $u\in\Q^s$, and we want to determine if the linear system 
\begin{equation}\label{originalsystem}NX=u\end{equation}
has some solution $v\in\Q^s$ that lies in $B$. 
To solve this problem, we will use a standard technique to solve these kind of systems in $\Z$.  The Smith normal form for $N$ is a diagonal matrix $D$ with diagonal entries $k_1,\ldots,k_r,0,\ldots,0$, such that $0<k_j$ for $0\leq j\leq r$ and each $k_j$ divides the next $k_{j+1}$, with $r$ being the rank of $N$. Moreover, there are invertible matrices $P$ and $Q$ in $\text{SL}(s,\Z)$ such that $D=QNP$. 

We set $$a=\text{max}\{|a_{ij}|\mid a_{ij}\text{ entry of }N\}.$$

\begin{lem}\label{smith1} Let $N$ be any integral $s\times s$ matrix and let $D=\text{diag}(k_1,\ldots,k_r,0,\ldots,0)$ be its Smith normal form. Then
$$k_1\ldots k_r\leq \sqrt{s}a^s$$
\end{lem} 
\begin{proof} It is well known that the product $k_1\ldots k_r$ is the greatest common divisor of the determinants of the nonsingular $r\times r$ minors of the matrix $N$. Let $N_1$ be one of those minors. Then
$$k_r\leq k_1\ldots k_r\leq|\text{det}N_1|.$$
Now, the determinant of the matrix $N_1$ is bounded by the product of the norms of the columns $c_1,\ldots,c_r$ of the matrix (this bound is due to Hadamard, see for example \cite{Horn1985}) so we have
$$|\text{det}N_1|\leq\prod_{j=1}^r\|c_j\|\leq\sqrt{r}^ra^r.$$
\end{proof}

Recall that we are assuming that $N$ commutes with all the matrices $M_l$. Under this assumption we claim that we can solve the problem above by using Lemma \ref{vinB}. To demonstrate, let $P$ and $Q$ be invertible matrices in $\text{SL}(s,\Z)$ such that $D=QNP=\text{diag}(k_1,\ldots,k_r,0,\ldots,0)$
is the Smith normal form of $N$. Our system can then be transformed into 
\begin{equation}\label{smithsystem} D\tilde X= \begin{scriptsize} \begin{pmatrix}
0&0  \\
    0&  D_2  \\
  \end{pmatrix} 
  \end{scriptsize}\tilde X=Qu
\end{equation}
with $\tilde X=P^{-1}X$. At this point, we see that the system has some solution if and only if the first $s-r$ entries of $Qu$ vanish. 
Assume that this is the case and let $v_2$ be the unique solution to the system 
\begin{equation}\label{smallsystem}D_2\tilde X_2=(Qu)_2\end{equation}
 where the subscript $2$ in $\tilde X$ and $Qu$ means that we take the last $r$ coordinates only. Then
 $$v_2=D_2^{-1}(Qu)_2.$$

 The set of all the rational solutions to (\ref{originalsystem}) is 
$$\Big\{P
  \begin{scriptsize}\begin{pmatrix}
v_1  \\
      v_2  \\
  \end{pmatrix} 
  \end{scriptsize}\mid v_1\in\Q^{s-r}
\Big\}.$$
Equivalently, this set can be written as $$v+\text{Ker}N\, \text{
where }\, v=P \begin{scriptsize}\begin{pmatrix}
0 \\
      v_2  \\
  \end{pmatrix} 
  \end{scriptsize}.$$

Observe that the columns of $P$ give a new basis of $\Z^s$ that can be used to define $B$ instead of $b_1,\ldots,b_s$. In this new basis the action of each $q_l$ is encoded by the matrix $P^{-1}M_lP$. The fact that $N$ commutes with each $M_l$ implies that $M_l$ leaves $\text{Ker}N$ (setwise) invariant. By construction,  $\text{Ker}N$ is generated by the first $s-r$ columns of $P$ and therefore each $P^{-1}M_lP$ has the following block upper triangular form:
$$P^{-1}M_lP= \begin{scriptsize}\begin{pmatrix}
A_l&B_l  \\
     0& C_l  \\
  \end{pmatrix} 
  \end{scriptsize}.$$
Moreover, $C_l$ is just the $r\times r$ matrix associated with the action of $q_l$ in the quotient $\Q^s/\text{Ker}N$, written in the basis obtained from the last $r$ columns of $P$.

\begin{prop}\label{solutionsystem} A solution to the system (\ref{smithsystem}) exists in $B$ if and only if $v_2\in\Z[{1\over d}]^r$ and 
$$C^{tr\lfloor\log d\rfloor}v_2\in\Z^r,$$
with $C=\prod_lC_l$ and $t$ the smallest possible integer such that $d^tv_2$ is integral. (We can use $s$ instead of $r$).
\end{prop}
\begin{proof} Assume first that $C^{tr\lfloor\log d\rfloor}v_2\in\Z^r$, with $t$ as above. We have
$$P^{-1}M^{t\alpha}P= \begin{scriptsize}\begin{pmatrix}
A&S  \\
     0& C^{tr\lfloor\log d\rfloor}\\
  \end{pmatrix} 
  \end{scriptsize}$$
for certain $(s-r)\times r$ matrix $S$ and certain $(s-r)\times (s-r)$ invertible matrix $A$, with $M=\prod_lM_l$ as before. Therefore
$$P^{-1}M^{tr\lfloor\log d\rfloor}P\tilde X= \begin{scriptsize}\begin{pmatrix}
A&S  \\
     0& C^{tr\lfloor\log d\rfloor}\\
  \end{pmatrix} \begin{pmatrix}v_1\\ v_2\end{pmatrix}=\begin{pmatrix}Av_1+Sv_2\\ C^{tr\lfloor\log d\rfloor}v_2\end{pmatrix}.
  \end{scriptsize}$$
  This means that now we only have to  find a $v_1\in\Q^{s-r}$ such that $Av_1+Sv_2\in\Z^s$. To do it, observe that it suffices to take $v_1=-A^{-1}Sv_2'$.

\

Conversely, assume that some $P\begin{scriptsize}\begin{pmatrix}v_1\\ v_2\end{pmatrix}
  \end{scriptsize}$ lies in $B$. Then some product of positive powers of the $M_l$'s transforms  $P\begin{scriptsize}\begin{pmatrix}v_1\\ v_2\end{pmatrix}
  \end{scriptsize}$ into an integral vector, thus there is a product of the $C_l$'s that transforms $v_2$ into an integral vector. 
  We may use now Lemma \ref{vinB} applied to $\Q^r=\Q^s/\text{Ker}N$ with respect to the action of the matrices $C_l$ to conclude that
  $v_2\in\Z[{1\over d}]^r$ and $$C^{tr\lfloor\log d\rfloor}v_2\in\Z^r,$$
with $t$ the smallest possible integer such that $d^tv_2$ is integral. (Note that $dC_l^{-1}$ is integral so we can use the same $d$ for this quotient as for the original group.)\\
\end{proof}

\begin{rem}\label{iforv2}{\rm Observe that, as $N$ is integral, a necessary condition for (\ref{originalsystem}) to have some solution in $B$ is that $u$ lies in $\Z[{1\over d}]^d$.
Let $t_0$ be such that $d^{t_0}u$ is integral. Then $d^{t_0}\text{det}(D_2)v_2$ is also integral. If $v_2$ lies in $\Z[{1\over d}]^r$, it means that for some $t_1$ such that $d^{t_1}\leq\text{det}(D_2)$, we have that $d^{t_0+t_1}v_2$ is integral. By Lemma \ref{smith1} $\text{det}(D_2)\leq \sqrt{s}a^s$, thus $t_1\leq\sqrt{s}a^s$. As a consequence, if $t$ is as in Proposition \ref{solutionsystem}, we have
$$t\leq t_0+\sqrt{s}a^s.$$
}\end{rem}

Now we are ready to show:

\begin{prop}\label{complexitysolutionsystem} There is an algorithm to decide whether the system (\ref{originalsystem}) has some solution in $B$ and to compute that solution. The complexity of this algorithm is polynomial,  specifically
$$
O(s^6\log sa+ (s-r)^5 +(s-r)^3 + (n-1)[s^3\log ts\log d+1]+r^3).
$$
where $t\leq t_0+\sqrt{s}a^s$ and $t_0$ is such that $d^{t_0}u$ is integral. (If there is no such $t_0$ then the system has no solution in $B$).
\end{prop}
\begin{proof}
The algorithm has been described above. In summary, we have to transform the original system using the Smith normal form for $N$, compute $v_2$ and the matrices $C_l$ and $C=C_1\ldots C_n$, and then check whether $v_2$ lies in $\Z[{1\over d}]^r$. If it does, we may either compute $t$ such that  $d^tv_2$ is integral or estimate $t$ as $t_0+t_1$ (see Remark \ref{iforv2}). Then we compute 
$$C^{tr\lfloor\log d\rfloor}v_2$$
and check whether it is integral or not.

\
 
To estimate the complexity of this procedure observe that for an integral matrix $N$, the time complexity of computing the Smith normal form $D$ and invertible integral matrices $P$ and $Q$ such that $QNP=D$  is polynomial, specifically $O(s^6\log sa)$, where $a$ is the maximum absolute value of the entries of $N$. For a proof of this fact see \cite{KB} in the non-singular case and \cite{CY} for the singular one. 
Once we have the Smith normal form, to compute $v_2$ we only have to perform the product of $D_2^{-1}$ and $(Qu)_2$ which has complexity $O(r^3)$. Next, we have to compute the matrices $C_l$, which requires $n-1$ matrix multiplications, thus $O((n-1)s^3)$. We then check whether $C^{tr\lfloor\log d\rfloor}v_2$ is integral which takes at most $O((n-1) s^3\log ts\log d)$ time.
Solving for $v_2$ and $v_1'$ via Gaussians elimination take $O(r^3)$ and $O((s-r)^3)$, respectively, and calculating $v_1$ is $O((s-r)^5$. The overall time complexity is then the sum of the above operations. Note that the lower order terms involving $s$ and $r$ are dominated by the complexity of calculating the Smith normal form.
\end{proof}

\section{On the Complexity of the Conjugacy Problem
}\label{complexity}

\subsection{The Conjugacy Search Problem in Split Metabelian Groups}\label{general} We begin with a few considerations about the conjugacy search problem in split metabelian groups not necessarily in $\mathcal{F}$. So for the moment we consider
a group $G$ of the form $G=B\rtimes Q$ with both  $B$ and $Q$ abelian groups. As before, we use multiplicative notation for the whole group $G$ but additive notation for $B$. 

\

Assume that we have conjugate elements $g,g_1\in G$ and we want to solve the conjugacy search problem for $g$, $g_1$, i.e., we want to find an
$h\in G$ such that 
$$g^h=g_1.$$
Let $g=bx$, $g_1=b_1x_1$ and $h=cy$ with $b,b_1,c\in B$, $x,x_1,y\in Q$, then
$$b_1x_1=g_1=g^h=hgh^{-1}=cybxy^{-1}c^{-1}=cb^y(c^{-1})^{x}x$$
thus $x=x_1$ and from now on we denote this element solely by $x$.
The element $cb^y(c^{-1})^{x}$ belongs to the abelian group $B$. We write it additively
$$c-x\cdot c+y\cdot b=y\cdot b+(1-x)\cdot c.$$
This means that the conjugacy search problem above 
is equivalent to the problem of finding $c\in B$, $y\in Q$ such that
\begin{equation}\label{equation}b_1=y\cdot b+(1-x)\cdot c\end{equation}
when $b,b_1\in B$ and $x\in Q$ are given.
This problem can be split into two parts:

\begin{itemize} 
\item[i)] find a $y\in Q$ such that $b_1\equiv y\cdot b$ modulo the subgroup $(1-x)B$ 

\item[ii)] find a $c\in B$ such that $(1-x)\cdot c=b_1-y\cdot b$ where $y$ is the element found in i).
\end{itemize}
The complexity of both problems depends on the groups $Q$ and $B$ and potentially on the choice of $x$, $b$ and $b_1$ as well.

\

We now return to the particular case when our group belongs to $\mathcal{F}$. In this case there is a decomposition $G=B\rtimes Q$ with $Q$ a free abelian group of finite rank and $B$ an additive subgroup of $\Q^s$. We fix the elements $x\in Q$, $b$ and $b_1\in B$. We also use the same notation for $G$ as in the previous section so we assume that $G$ has a presentation as (\ref{F}), that $M_1,\ldots,M_n$ are the matrices encoded in that presentation that yield the action of the $q_1,\ldots,q_n$ and that $d$ is an integer so that $dM_l^{-1}$ is an integer matrix for each $l=1,\ldots,n$.


\

We let $M_x$ be the rational matrix associated with the action of $x$ on $B$, that is, if $x=q_1^{\alpha_1}\ldots q_n^{\alpha_n}$, then
$$M_x=M_1^{\alpha_1}\ldots M_n^{\alpha_n}.$$
For technical reasons it will be useful to have a special notation for the matrix $I-M_x$. So we put  $N_x=I-M_x$.
Consider the $Q$-invariant subgroup of $B$
$$N_xB=(1-x)\cdot B.$$
 There is an induced action of $Q$ on the quotient group $\bar{B}=B/((1-x)\cdot B)=B/N_xB$. We use $\bar{\quad}$ to denote the coset in $\bar{B}$ associated with a given element. With this notation equation in i) above is
$$\bar b_1=y\cdot \bar b.$$

Let $T$ be the torsion subgroup of $\bar B$. This subgroup is also $Q$-invariant so we get an induced $Q$-action on  $\bar B/T$.
The fact that $\bar B/T$ is torsion-free and of finite Pr\"ufer rank implies that it can be embedded in $\Q^{s_1}$ for some $s_1$. Explicitly, as  $\Q$ is flat we get
$$\bar B/T\hookrightarrow \bar B/T\otimes\Q=(B/N_xB)\otimes\Q=(B\otimes\Q)/(N_xB\otimes\Q)=\Q^s/N_x\Q^s$$
and one can find the matrices encoding the action of each of the elements $q_l$  on $\bar{B}/T$.

\

As we will see below, problem i) can be reduced by quotienting out the subgroup $T$ to a multiple orbit problem in a vector space and problem ii) is a type of discrete log problem. For the first we take advantage of the polynomial time solution of the multiple orbit problem in a vector space given in  \cite{aszl1996multiplicative}. The latter can be solved using the fact, proved in the next subsection, that $T$ is finite. Moreover, we will see that one can get an upper bound for the complexity of this algorithm that is essentially dependent upon the size of  $T$. 

\subsection{On the Subgroup $T$} We proceed to showing that $T$ is indeed finite, and calculating a bound for its size. 
Recall that the exponent of a torsion abelian group $T$, denoted $\text{exp}(T)$, is the smallest non-negative integer $k$ such that $kv=0$ for any $v\in T$. (If there is no such integer, then the exponent is infinite). The following lemma is well known, but we include it here for completeness:

\begin{lem}\label{abelianprufer} Let $T$ be a torsion abelian group of finite Pr\"ufer rank $s$. Assume that $k=\text{exp}(T)<\infty$. Then $T$ is finite and
$$|T|\leq k^s.$$ 
\end{lem}
\begin{proof} Observe that as $T$ has finite exponent, its $p$-primary component $T_p$ vanishes for all primes $p$ except of possibly those primes dividing $k$. Moreover, $T$ cannot contain quasicyclic groups $C_{p^\infty}$. Then, using \cite[5.1.2]{robinson2004} (see also item 3 in page 85), we see that for any prime $p$ dividing $k$, $T_p$ is a sum of at most $s$ copies of a cyclic group of order at most the $p$-part of $k$. 
As $T=\oplus_{p\mid k}T_p$ we deduce the result.
\end{proof}

\begin{lem}\label{integer} Let $N$ be a square $s\times s$ integer matrix and $T$ the torsion subgroup of the group $\Z^s/N\Z^s$. Then
$$\text{exp}(T)\leq \sqrt{s}a^s$$
with $$a=\text{max}\{|a_{ij}|\mid a_{ij}\text{ entry of }N\}.$$ 
\end{lem}
\begin{proof}  Let  $D=\text{diag}(k_1,\ldots,k_r,0,\ldots,0)$ be the Smith normal form of $N$. Then
$$\text{exp}(T)=k_r\leq k_1,\ldots,k_r$$
so it suffices to apply Lemma \ref{smith1}. 

\end{proof}

\begin{thm}\label{bound}  Let $T$ be the torsion subgroup of the abelian group
$\bar{B}={B/(1-x)\cdot B}$.
Then $T$ is finite and  
$$|T|\leq \sqrt{s}^sd^{\mathcal{L}s^2}(a+1)^{s^2}$$
where $\mathcal{L}$ is the length of the element $x$ as a word in the generators of $Q$, $a$ is the maximum absolute value of an entry in $M_x$, the matrix associated with the action of $x$ on $B$.

\end{thm} 
\begin{proof} Let $N_x=I-M_x$.  Assume first  that $M_x$ is an integral matrix, so the same happens with $N_x$. We want to  relate the exponent of $T$ with the exponent of the torsion subgroup of $\Z^s/N_x\Z^s$. Let $k$ be this last exponent and choose $\beta\in B$ such that $1\neq\bar \beta$ lies in $T$. Denote by $m>0$ the order of $\bar \beta$. Observe that $m\beta=N_x\gamma$ for some $\gamma\in B$ and that $m$ is the smallest possible under these conditions. 

\

Next, choose $q\in Q$ such that $q\cdot \beta$ and $q\cdot \gamma$ both lie in $\Z^s$. To find such a $q$ it suffices to
 write $\beta$ and $\gamma$ multiplicatively as in (\ref{form}) and take as $q$ a product of the $q_l$'s with big enough exponents.

\

Then we have $m(q\cdot \beta)=q\cdot N_x\gamma=N_x(q\cdot \gamma)\in N_x\Z^s$ thus $q\cdot b\beta+N_x\Z^s$ lies in the torsion subgroup of $\Z^s/N_x\Z^s$. Therefore, $k(q\cdot \beta)\in N_x\Z^s$. Now, let $m_1$ be the greatest common divisor of $m$ and $k$ and observe that the previous equations imply $m_1(q\cdot \beta)\in N_x\Z^s$. This means that for some $\gamma_1\in\Z^s$ we have
$m_1(q\cdot \beta)=N_x\gamma_1$, thus
$$m_1\beta=q^{-1}N_x\gamma_1=N_xq^{-1}\gamma_1=N_x\gamma_2$$
with $\gamma_2=q^{-1}\cdot \gamma_1\in B$. By the minimality of $m$ we must have $m\leq m_1$. As $m_1$ divides both $k$ and $m$ we can conclude $m=m_1\mid k$. This implies that $k$ is also the exponent of $T$. 

\
 
Next, we consider the general case when $N_x$ could be non-integral. 
As $M_x$ is the product of $\mathcal{L}$ matrices in the set $\{M_1^{\pm1},\ldots,M_n^{\pm}\}$ we see that the matrix $d^{\mathcal{L}}M_x$ is integral and therefore so is $d^{\mathcal{L}}N_x$. 
Obviously, the group $N_xB/d^{\mathcal{L}}N_xB$ is torsion thus
$$\text{exp}(T)\leq\text{exp}(\text{torsion subgroup of }B/d^\mathcal{L}N_xB).$$
The matrix $d^{\mathcal{L}}N_x$ also commutes with the $Q$-action so what we did above implies that this last exponent equals the exponent of the torsion subgroup of $\Z^s/d^{\mathcal{L}}N_x\Z^s$. From all this together with Lemma \ref{integer} and using that the biggest absolute value of an entry of $d^\mathcal{L}$ is bounded by $d^\mathcal{L}N_x$ we get
$$\text{exp}(T)\leq \sqrt{s}d^{{\mathcal{L}}s}(a+1)^s.$$
 
Finally, as the group $\bar B$ has finite Pr\"ufer rank, so does $T$, therefore by Lemma \ref{abelianprufer} we get the result. 
 \end{proof}

 \

\begin{rem}{\rm The maximum absolute value of an entry in the matrix $M_x$ is bounded exponentially on $\mathcal{L}$. Therefore, its logarithm is bounded linearly on $\mathcal{L}$. To see it, observe first that 
if $M_1$ and $M_2$ are $s\times s$ matrices and $\mu$ is an upper bound for the absolute value of the entries of both $M_1$ and $M_2$, then the maximum absolute value of an entry in the product $M_1M_2$ is bounded by
$s\mu^2.$ Repeating this argument one sees that if $x$ has length $\mathcal{L}$ as a word in $q_1,\ldots,q_n$ and $\mu$ is an upper bound for the absolute value of the entries of each $M_l$, then the maximum absolute value $a$ of an entry of $M_x$ is bounded by
$$s^{\mathcal{L}-1}\mu^\mathcal{L}$$
}
\end{rem}

\

The next result yields a bound on the order of $T$ which is exponential in the length $\mathcal{L}$ of $x$.

\begin{prop}\label{Texponential} With the previous notation, there is a constant $K$, depending on $G$ only such that for $T$ the torsion subgroup of $\bar{B}=B/N_xB$, 
$$|T|\leq K^\mathcal{L}$$
where $\mathcal{L}$ is the length of $x$.
\end{prop}
\begin{proof}  Let $\mu$ be an upper bound for the absolute value of the entries of each $M_l$. By Theorem \ref{bound} and the observation above
$$|T|\leq \sqrt{s}^sd^{\mathcal{L}s^2}(a+1)^{s^2}\leq \sqrt{s}^sd^{\mathcal{L}s^2}(s^{\mathcal{L}-1}\mu^\mathcal{L}+1)^{s^2}\leq(\sqrt{s}ds\mu+\sqrt{s}d)^{s^2\mathcal{L}}$$
so we only have to take $K=(\sqrt{s}ds\mu+\sqrt{s}d)^{s^2}$.
\end{proof}

\

\

\subsection{Description of the Algorithm}\label{algorithm}

Recall that to find an element that conjugates $bx$ into $b_1x$ (here, $x\in Q$, $b,b_1\in B$) we need to find:
\begin{itemize}
\item{i)} $y\in Q$ such that $\bar{b_1}=y\cdot\bar{b}$ where $\bar{}$ means passing to $\bar{B}=B/N_xB$,

\item{ii)} $c\in B$ such that $N_xc=b_1-y\cdot b$.
\end{itemize}

\

\noindent{\bf Step 1 (Problem i):} With $M_x$ and $N_x$ as before, form the quotient $V=\Q^s/N_x\Q^s$ and find the matrices encoding the action of each $q_l$ on $V$.
Consider the projections $\bar b+T$ and $\bar b_1+T$ of $b$ and $b_1$ in $\bar B/T$ and see them as elements in $V$ (via the embedding $\bar B/T\hookrightarrow V$). Then use the algorithm in  \cite{aszl1996multiplicative} to solve the multiple orbit problem
$$y\cdot (\bar b+T)=\bar b_1+T.$$
This algorithm determines not only a single $y$ but the full lattice of solutions.
$$\Lambda=\{q\in Q\mid q\cdot\bar b-\bar b_1\in T\},$$
Furthermore, it allows one to compute a basis $y_1,\ldots,y_m$ of the following subgroup of $Q$
$$Q_1=C_Q(\bar b+T)=\{q\in Q\mid q\cdot \bar b-\bar b\in T\}.$$ 
(Note that for any $z\in\Lambda$,  $Q_1=z^{-1}\Lambda$).

\

\noindent{\bf Step 2 (Problem ii):} Order the elements of $Q_1$ according to word length. For each $q\in Q_1$ check whether $q\cdot b-b_1\in N_xB$. Each check consists of trying to solve a system of linear equations. More precisely, we have to check whether the system
$$u= N_xX$$
with $u=q\cdot b-b_1$ has some solution in $B$. This can be done using Proposition \ref{solutionsystem}. When found, the solution yields the element $c$ of $B$ that we needed.

\

Of course, a priori this procedure might never halt. But as we shall see next this is not the case, for the number of iterations in Step 2 is bounded by the size of the group $T$, which has been shown to be finite. 

\

We can now be more explicit. Recall that the problem is to find a $y\in Q$ such that $y\cdot \bar b=\bar b_1$, and, as this is the search variant of the conjugacy, $y$ exists. We will use the notation above so we know that all the solutions $y$ lie in the set
$$\Lambda=\{q\in Q\mid q\cdot\bar b-\bar b_1\in T\}.$$
Again let $Q_1=C_Q(\bar b+T)=\{q\in Q\mid q\cdot \bar b-\bar b\in T\}$.  Let  $z\in\Lambda$ be an arbitrary element that we will assume fixed from now.  Then we have  $\Lambda=zQ_1$ thus
 for any $q\in Q_1$, the element $zq\cdot\bar b-\bar b_1$ lies in $T$ and as $T$ is finite there are only finitely many possibilities for its value. Moreover, we know that eventually it takes the value 0.
Put
$$Q_2=C_Q(\bar b)=\{q\in Q\mid q\cdot \bar b=\bar b\}=\{q\in Q\mid q\cdot b- b\in N_xB\}.$$
We obviously have $Q_2\leq Q_1$ and for $q_1,q_2\in Q_1$,
$$zq_1\cdot\bar b-\bar b_1=zq_2\cdot\bar b-\bar b_1\text{ if and only if }q_1Q_2=q_2Q_2.$$ As $T$ is finite we conclude that
the quotient $Q_1/Q_2$ is of finite order bounded by $|T|$. If $\{y_1\ldots,y_{|T|}\}$ is a set of representatives of the cosets of $Q_2$ in $Q_1$, then some element $y$ in the finite set
$$\{zy_1,\ldots,zy_{|T|}\}$$ 
is the $y\in Q$ that satisfies $y\cdot \bar b=\bar b_1$.

\

In the next lemma we prove that by $Q_1$ being a lattice we can produce a full set of representatives as before, including our $y$, by taking elements solely from $Q_1$, Moreover, the number of steps needed is bounded in terms of $|T|$.

\begin{lem} Let $Q_2\leq Q_1$ with $Q_1$ free abelian with generators $x_1,\ldots,x_m$, and assume that the group $Q_1/Q_2$ is finite. Then the set
$$\Omega=\{x_1^{\alpha_1}\ldots x_m^{\alpha_m}\mid \sum_{j=1}^m|\alpha_j|< |Q_1/Q_2|\}$$
has order bounded by $(2 |Q_1/Q_2|)^{m}$ and contains a full set of representatives of the cosets of $Q_2$ in $Q_1$. \end{lem}
\begin{proof} Let $v_1,\ldots,v_m$ be generators of the subgroup $Q_2$, which can be viewed as points in $\Z^m$. Consider the parallelogram
$$P=\{t_1v_1+\ldots+t_mv_m\mid t_j\in\R, 0\leq t_j<1\}.$$
Then $\Z^m\cap P$ is a set of representatives of the cosets of $Q_2$ in $Q_1$ and we claim that $P\subseteq\Omega$. Observe that for any point $p=(\alpha_1,\ldots,\alpha_m)$ in $\Z^m\cap P$ there is a path in $\Z^m\cap P$ from $(0,\ldots,0)$ to $p$. We may assume that the path is simple and therefore its length is bounded by $ |Q_1/Q_2|$. On the other hand, the length of the path is greater than or equal to $\sum_{j=1}^m|\alpha_j|$ thus
$$\sum_{j=1}^m|\alpha_j|\leq  |Q_1/Q_2|.$$
\end{proof}

The number of iterations of Step 2 is bounded by the value $|Q_1/Q_2|$.
At this point, it is clear that smaller groups $Q_1/Q_2$ will reduce the running time of the algorithm. Observe that by construction, the element $x$ belongs to the group $Q_2$. In the case when $Q$ is cyclic this yields a dramatic improvement of our bound for $|Q_1/Q_2|$: we only have one generator, say $q_1$ of $Q$, thus, if $x=q_1^\mathcal{L}$, $|Q_1/Q_2|\leq  |Q/Q_2|=\mathcal{L}$. Moreover, in this case Step 1 in our algorithm is not needed, so we only have to perform $\mathcal{L}$ iterations of Step 2, and our algorithm coincides with the one in \cite{cavallo2014polynomial}.
\goodbreak


\subsection{Complexity Analysis and Consequences} We can now prove Theorem \ref{exponential}:

\noindent{{\sl Proof of Theorem \ref{exponential}}.
We consider the complexity of the algorithm of Subsection \ref{algorithm}. We assume that $g$ and $g_1$ are given as words as in (\ref{form}).
Observe that Step 1 only requires polynomial time. As for Step 2, we have to consider an exponential (in $\mathcal{L}$) number of systems of linear equations of the form
$$u= N_xX$$
with $u=q\cdot b-b_1$. Moreover, we may find (by writing $u$ as in (\ref{form})) some $q\in Q$ such that $q\cdot u$ is in the group generated by $b_1\ldots,b_s$. If $R$ is the matrix representing the action of $q$, this is equivalent to the vector $Ru$ being integral. As $R$ and $N_x$ commute our system can be transformed into
$$N_xRX=Ru.$$
Obviously, $X$ lies in $B$ if and only if $RX$ does, thus the problem is equivalent to deciding whether 
$$d^\mathcal{L}N_xX_1=d^\mathcal{L}Ru$$
has some solution $X_1$ in $B$.

\

Using Proposition \ref{solutionsystem} and the complexity computation of Proposition \ref{complexitysolutionsystem} we see that this can be done in a time that is polynomial on log of the maximum absolute value of an entry in $d^\mathcal{L}N_x$. Observe that our integrality assumption on $Ru$ implies that the integer denoted $t_0$ in Proposition \ref{complexitysolutionsystem} can be taken to be 0. As  the maximum absolute value of an entry in $d^\mathcal{L}N_x$ is exponential on $\mathcal{L}$, this time is polynomial on $\mathcal{L}$. The exponential bound in the result then follows because we are doing this a number of times which is exponential on $\mathcal{L}$.
\qed

\

Next, we consider a particular case in which the running time of the algorithm is reduced to polynomial with respect to the length $\mathcal{L}$ of $x$.

Let $s_1,s_2\geq 0$ be integers with $s=s_1+s_2$ and denote

$$\Gamma_{s_1,s_2}:=  \left\{ \text{Matrices }  
\begin{scriptsize}
  \begin{pmatrix}
 I_{s_1} &   A  \\
      0 & I_{s_2} \\
  \end{pmatrix} 
  \end{scriptsize}
 \right\}\leq SL(s,\Z).$$
  Observe that if we consider a group $G$ with a presentation as in (\ref{F}) such that the matrices $M_l$ lie in $\Gamma_{s_1,s_2}$, 
  then each $M_l^{-1}$ is also an integral matrix so we can choose $d=1$, where $d$ is the integer of Theorem \ref{bound}. This implies that this group is polycyclic. Moreover, it is easy to see that it is in fact nilpotent.

\begin{prop} With the previous notation, assume that for $l=1,\ldots,n$,
$$M_l\in\Gamma_{s_1,s_2}.$$ 
Then there is some constant $K$ depending on $G$ only such that for $T$, the torsion subgroup of $B/N_xB=(1-M_x)B$, 
$$|T|\leq K\mathcal{L}^{s^2}$$
where $\mathcal{L}$ is the length of $x$.
\end{prop}
\begin{proof} We consider the bound of Theorem \ref{bound} for $d=1$ (see above) 
$$|T|\leq \sqrt{s}(a+1)^{s^2},$$
where $a$ is the maximum absolute value of an entry in $M_x$. Observe that $M_x$ is a product of matrices in $\Gamma_{s_1,s_2}$ and that
$$ \begin{pmatrix}
 I_{s_1} &   A_1  \\
      0 & I_{s_2} \\
  \end{pmatrix} 
 \begin{pmatrix}
 I_{s_1} &   A_2  \\
      0 & I_{s_2} \\
  \end{pmatrix} =
   \begin{pmatrix}
 I_{s_1} &   A_1+A_2  \\
      0 & I_{s_2} \\
  \end{pmatrix} .
$$
Therefore, if we let $\mu$ be the maximum absolute value of an entry in each of the matrices $A_1,\ldots,A_n$, then
$a\leq \mathcal{L}\mu$ and therefore
$$|T|\leq \sqrt{s}(a+1)^{s^2}\leq \sqrt{s}(\mathcal{L}\mu+1)^{s^2}\leq  \sqrt{s}(2\mathcal{L}\mu)^{s^2}$$
so it suffices to take $K=\sqrt{s}(2\mu)^{s^2}$.
\end{proof}

This result together with the algorithm above (recall that $d=1$ in this case) imply the following:

\begin{thm}\label{polynomial}   With the previous notation, assume that for $l=1,\ldots,n$,
$$M_l\in\Gamma_{s_1,s_2}.$$ 
then the complexity of the conjugacy problem in $G$ is at most polynomial. 
\end{thm}

We finish this section with a remark on conjugator lengths. Let $g$ and $g_1$ be conjugate elements in $G$. Our algorithm primarily consists of identifying a suitable subgroup $Q_1$ of $Q$ and showing that, for a function dependent upon the length $\mathcal{L}$ of $x$, there exists some $y\in Q_1$ whose length is bounded by that function and which is the $Q$-component of an element $h$ such that $g^h=g_1$. Essentially, we are providing an estimation for the $Q$-conjugator length function.
We make this more precise in the next result.

\begin{cor}\label{length} There exists an integer $K>0$ dependent upon $G$ only such that for any conjugate elements $g,g_1\in G$,  with $g=bx$, $g_1=b_1x$ for $x\in Q$ and $b,b_1\in B$, there is some $h=cy$ for $c\in B$, $y\in Q$ and $g^h=g_1$ such that 
the length of $y$ is bounded by $K^\mathcal{L}$, where $\mathcal{L}$ is the length of $x$. In the particular case when $Q\leq\Gamma_{s_1+s_2}$, 
 the length of $y$ is bounded by $K\mathcal{L}^{s^2}$.
\end{cor}

\subsection{Reduction to the Discrete Logarithm Problem}\label{discretelog} For this subsection, we restrict ourselves to the situation of Example \ref{galois} where $Q$ is a multiplicative subgroup of a field $L$ such that $L:\Q$ is a Galois extension and $B$ is the additive group  $\mathcal{O}_L[q_1^{\pm},\ldots,q_n^{\pm}]$ which is sandwiched between $\Q$ and $L$. In particular, this means that the only element in $Q$ with an associated matrix having an eigenvalue of 1 is the identity matrix: the eigenvalues of the matrix representing an element $q\in L$ are precisely $q$ itself and its Galois conjugates and thus cannot be 1 if $q\neq 1$. Recall also that Example \ref{galois} includes Example \ref{genBS}.

\

We will keep the notation of the previous sections, with elements $bx$, $b_1x\in G$ such that there is some $cy\in G$ with (additively)
$$b_1=y\cdot b+(1-x)\cdot c.$$

We may consider $y$ and $1-x$ as elements in the field $L$. From now on we omit the $\cdot$ from our notation and use juxtaposition to denote the action. Now, $B$ also has a ring structure and $(1-x)B$ is an ideal in $B$. Moreover, in this case the quotient ring $\bar{B}=B/(1-x)B$ is finite (because the matrix associated with $1-x$ is regular.) In this finite quotient ring we wish to solve the equation
$$y\bar{b}=\bar{b_1}.$$
  Let $y=q_1^{t_1}\ldots q_k^{t_k}$, then solving the discrete log problem in $\bar{B}=B/(1-x)B$ consists of finding $t_1,\ldots,t_k$ so that
  $$q_1^{t_1}\ldots q_k^{t_k}\bar{b}=\bar{b_1}$$
  in the finite ring $\bar{B}$.
   
\
   
This is a special type of discrete log problem as one can observe by recalling what happens when $Q$ is cyclic: then $x=q_1^s$ for some $s$ thus we have to solve
$$q_1^{t_1}\bar{v}=\bar{w}$$
in $\bar{B}=B/(1-q_1^s)B$. To solve it $s$ trials are sufficient (see \cite{cavallo2014polynomial}).

Let us restrict ourselves further to the case of generalized Baumslag-Solitar groups (i.e., the groups of Example \ref{genBS}.)
We identify the elements $q_l$ with the integers $m_l$ encoding their action. Assume that each $m_l$ is coprime with $1-m_1$. As before let $y=m_1^{t_1}\ldots m_k^{t_k}$ and choose $x=m_1$. 
Then as each $m_l$ is coprime with $1-m_1$,
$$\bar{B}=\Z[m_1^{\pm},\ldots,m_k^{\pm}]/(1-x)\Z[m_1^{\pm},\ldots,m_k^{\pm}]=\Z/(1-x)\Z=\Z_{1-m_1}.$$
We then have to find $t_2,\ldots,t_k$ such that
 $$m_2^{t_2}\ldots m_k^{t_k}\bar{b}=\bar{b_1}$$
in the  ring of integers modulo $1-m_1$. If $k=2$ this is an instance of the ordinary discrete logarithm problem.

\section*{Acknowledgements}
We thank Bren Cavallo who helped us in the beginning stage of this paper.\\
Delaram Kahrobaei is partially supported by a PSC-CUNY grant from the CUNY Research Foundation, the City Tech Foundation, and ONR (Office of Naval Research) grants N000141210758 and N00014-15-1-2164. Conchita Mart\'inez-P\'erez was supported by Gobierno de Arag\'on, European Regional 
Development Funds and partially supported by
MTM2015-67781-P (MINECO/FEDER) 
\bibliographystyle{plain}

\bibliography{lbabib}
\end{document}